\documentclass{amsart}

\usepackage{amsfonts, amsmath, amsthm}
\usepackage{epsfig, psfrag}

\theoremstyle{definition}

\theoremstyle{remark}

\numberwithin{equation}{section}

\swapnumbers \theoremstyle{plain}

\newtheorem{thm}{Theorem}[section]

\newtheorem{lem}[thm]{Lemma}

\newtheorem{cor}[thm]{Corollary}

\newtheorem{prop}[thm]{Proposition}

\theoremstyle{remark}

\theoremstyle{definition}
\newtheorem{defn}[thm]{Definition}

\newcommand{\T}{\mathcal T}
\newcommand{\C}{\mathcal C}

\renewcommand{\S}{{\mathcal S}}

\newcommand{\bdy}{\partial}
\newcommand{\bbb}{\mathbb}
\newcommand{\rppp}{\mathbb{R}P^3}

\newcommand{\td}{\tilde}
\newcommand{\open}[1]{\stackrel{\circ}{#1}}
\newcommand{\ve}{\varepsilon}

\newcommand{\abs}[1]{\lvert#1\rvert}

\begin{document}

\title{Annular-efficient triangulations of $3$--manifolds}
\author{William Jaco}
\address{Department of Mathematics, Oklahoma State University,
Stillwater, OK 74078}

\email{jaco@math.okstate.edu}
\thanks{The first author was partially supported by NSF/DMS Grants,
The Grayce B. Kerr Foundation, The American Institute of Mathematics
(AIM), and and The Visiting Research Scholar Program at University
of Melbourne (Australia)}

\author{J.~Hyam Rubinstein}
\address{Department of Mathematicss and Statistics,
University of Melbourne, Parkville, 
VIC 3010, Australia}
\email{rubin@ms.unimelb.edu.au}
\thanks{The second author was partially supported by The Australian Research
Council and The Grayce B. Kerr Foundation.}

%    General info
 \date{\today}


\begin{abstract} A triangulation of a compact 3-manifold is annular-efficient if it is 0--efficient and the only normal, incompressible annuli are thin edge-linking.  If a compact 3--manifold has an annular--efficient triangulation, then it is irreducible, $\bdy$--irreducible, and an-annular. Conversely, it is shown that for a compact, irreducible, $\bdy$--irreducible, and an-annular 3--manifold, any triangulation can be modified to an annular-efficient triangulation.  It follows that for a manifold satisfying this hypothesis, there are only a finite number of boundary slopes for incompressible and $\bdy$--incompressible surfaces of a bounded Euler characteristic.

\end{abstract}

\maketitle
%\tableofcontents
\section{Introduction}  In this paper we connect interesting properties of ideal triangulations of the interiors of compact 3--manifolds with interesting properties of triangulations of the compact 3--manifold by exploiting the inverse relationship between crushing a triangulation along a normal surface \cite{jac-rub-0eff}  and that of inflating an ideal triangulation \cite {jac-rub-inflate}.  In \cite{jac-rub-0eff} it is shown that a compact, irreducible, $\bdy$--irreducible, and an-annular 3-manifold admits an ideal triangulation of its interior. Here we show that any triangulation of such a $3$--manifold can be modified to an ideal triangulation of the  interior of the manifold; hence, providing a construction for such ideal triangulations.  In \cite{jac-rub-0eff} it also is shown that one can get 0--efficient ideal triangulations of  the interiors of these manifolds.  Here we show that we actually can construct ideal triangulations that satisfy a stronger condition that implies 0-efficient.

Marc Lackenby proves in \cite{lack-taut} that under our hypothesis and with the additional condition that each boundary component of the manifold is a torus, then these manifolds admit  a taut ideal triangulation of their interiors.  We were not able to show our construction will give taut ideal triangulations in the case of tori boundaries; indeed, at this point and in the case of tori boundaries, tautness  appears to be a stronger condition on an ideal triangulation than those we have.  It is shown in  \cite {kang-rub-ideal-II} that if in addition the manifold is atoridal, then a taut ideal triangulation is not only 0--efficient but is also 1--efficient in the sense that the only normal tori are vertex-linking.

In Section 2, we define what we mean for a triangulation of a compact 3--manifold with boundary to have a normal boundary and the notion of a normal surface being normally isotopic into the boundary.  Theorem \ref{determine-bdry-linking} establishes  for any triangulation of the manifolds we are interested in, there is an algorithm to decide if there is a closed normal surface that is isotopic into the boundary  but is not normally isotopic into the boundary. Furthermore, if there is one the algorithm will construct one.  In \cite{jac-rub-0eff} it is shown that given any triangulation, it can be decided if the triangulation is 0--efficient.  In Proposition \ref{decide-annular-eff}  we show the analogous result that it can be decided if a triangulation is annular-efficient. Furthermore, if it is not annular-efficient, the algorithm constructs a non vertex-linking normal 2--sphere or disk, or if there are none of these, it then constructs a non thin edge-linking normal annulus.  A triangulation having the property that any normal surface isotopic into the boundary must be normally isotopic into the boundary can be regarded as a $\bdy$--efficient condition in the sense we use 0-efficient, annular-efficient, and 1--efficient.

In Section 3 we review crushing triangulations along normal surfaces and inflating ideal triangulations.   An inflation is defined in terms of a combinatorial crushing; this enables us to establish very strong relationships between ideal triangulations and their inflations.  In particular, in Theorem \ref{bijection-ideal-inflate}  we establish a one-one correspondence between the closed normal
surfaces  in an ideal triangulation  and the closed normal surface in
any of its inflations; furthermore, we show corresponding surfaces under this correspondence are homeomorphic.

In Section 4 we have our main result: 
\vspace{.125 in}

\noindent{\bf Theorem 4.5.} \emph{ Suppose $M\ne \mathbb{B}^3$ is a compact,
irreducible, $\bdy$--irreducible,  an-annular $3$--manifold with
nonempty boundary. Then there is an algorithm that will modify any triangulation of $M$  to
an annular-efficient triangulation of $M$.}

\vspace{,125 in}We use annular-efficient triangulations to show that for a compact, irreducible, $\bdy$--irreducible, and an-annular 3-manifold there are only a finite number of boundary slopes possible for incompressible and $\bdy$--incompressible surfaces having a bounded Euler characteristic.  It has been communicated to us by David Bachman and Saul Schleimer that they have independently obtained a similar result.

\section{Triangulations and Normal Surfaces} We continue with the
use of (pseudo) triangulations and ideal triangulations as in  \cite{jac-rub-0eff}.

 If
$\mathbf{\td{\Delta}}$ is a pairwise disjoint collection of oriented
tetrahedra and $\mathbf{\Phi}$ is a family of orientation-reversing
face identifications of the tetrahedra in $\mathbf{\td{\Delta}}$,
then the identification space $X
=\mathbf{\td{\Delta}}/\mathbf{\Phi}$ is a 3-complex and is a
3--manifold at each point except possibly at the vertices. If $X$ is
a manifold, we denote the collection of tetrahedra and the face
identifications by a single symbol $\T$ and say $\T$ is a
\emph{triangulation} of the manifold $X$. If $X$ is not a manifold,
then $X\setminus\{vertices\}$ is the interior of a compact
$3$--manifold  with boundary, $M$, and we say $\T$ is an \emph{ideal
triangulation} of $\open{M}$, the interior of $M$; in this case we also say that $X$ is a \emph{pseudo-manifold} and $\T$ is an ideal triangulation of $X$. For an ideal triangulation, the
image of a vertex of a tetrahedron in $\mathbf{\td{\Delta}}$ is
called an {\it ideal vertex} and its {\it index} is the genus of its
vertex-linking surface. For an ideal triangulation we will always
assume the index of each vertex is $\ge 1$.

For our triangulations, the simplices of $\mathbf{\td{\Delta}}$ are
not necessarily embedded in $X$; however, the interior of each
simplex is embedded. We call the image in $X$ of a tetrahedron,
face, or edge in $\mathbf{\td{\Delta}}$, a tetrahedron, face, or
edge. For a tetrahedron $\Delta$ in $X$, there is precisely one
tetrahedron $\td{\Delta}$ in $\mathbf{\td{\Delta}}$ that projects to
$\Delta$, called the {\it lift of $\Delta$}. For a face $\sigma$ in
$X$, there are either one or two faces in $\mathbf{\td{\Delta}}$
that project to $\sigma$; if only one face projects to $\sigma$,
then $\sigma$ is in the boundary of $M$. If $e$ is an edge in $X$,
the number of edges in $\mathbf{\td{\Delta}}$ that project to $e$ is
the {\it index of $e$}.

See \cite{jac-rub-0eff} for more details regarding triangulations from
our point of view.

\subsection{Normal surfaces} If $M$ is a $3$--manifold and $\T$ is a
triangulation of $M$, we say the properly embedded surface $S$ in
$M$ is {\it normal} (with respect to $\T$) if for every tetrahedron
$\Delta$ in $\mathbf{\td{\Delta}}/\mathbf{\Phi}$,  the intersection
of $S$ with $\Delta$ lifts to a collection of normal triangles and
normal quadrilaterals in $\td{\Delta}$, the lift of $\Delta$.
% Similarly, we say the properly embedded surface $S$ in
%$M$ is {\it almost normal} (with respect to $\T$) if for every tetrahedron
 %in $\mathbf{\td{\Delta}}/\mathbf{\Phi}$,  the intersection
%of $S$ with the tetrahedron lifts to a collection of normal triangles and
%normal quadrilaterals, except for precisely one tetrahedron $\Delta$  in $\mathbf{\td{\Delta}}/\mathbf{\Phi}$ where the intersection of $S$  with $\Delta$ includes an almost normal octagon or an almost normal tube
%in $\td{\Delta}$, the lift of $\Delta$. 
Note
that since our tetrahedra have possible face identifications, the
intersection of a normal surface 
%or almost normal surface 
with a tetrahedron need not be a
normal triangle or a normal quadrilateral 
%an almost normal octagon, or an almost normal tube 
but might be one of these
with edge identifications.

We shall assume the reader is familiar with classical normal 
%and almost normal 
surface
theory, which carries over in all of our situations.

A triangulation of a compact $3$--manifold
with boundary is said to be a {\it normal boundary triangulation} or
to have a {\it normal boundary} if the frontier of a small regular
neighborhood of the boundary is normally isotopic to a normal
surface. In this case, we call the normal surface consisting of the
frontier of a small regular neighborhood of the boundary the {\it
normal boundary}. Not all triangulations have a normal boundary; for
example, layered triangulations of handlebodies \cite{jac-rub-layered} contain no closed normal surfaces
and, hence, can not have a normal boundary.  

A properly embedded surface in a compact 3--manifold with boundary is said to be \emph{isotopic into $\bdy M$} if there is an isotopy of the surface through $M$ into $\bdy M$ keeping the boundary of the surface fixed.  If the manifold is triangulated and the surface is  closed and normal, it is said to be \emph{normally isotopic into $\bdy M$}, if  the triangulation has normal boundary and the surface is normally isotopic to the normal boundary. We are interested in triangulations in which the only closed, normal surface isotopic into the boundary is the normal boundary. 

A properly embedded annulus in a $3$--manifold is {\it essential} if
it is incompressible and not isotopic into the boundary.  A compact
$3$--manifold is said to be {\it an-annular} if it has no properly
embedded, essential annuli.  

The following theorem gives conditions under which we can decide if a triangulation of a compact 3-manifold with boundary has a closed, normal surface that is isotopic into the boundary but is not normally isotopic into the boundary.  
 
\begin{thm} \label{determine-bdry-linking} Suppose $M$ is a compact, orientable 3--manifold with boundary that is irreducible,  $\bdy$--irreducible, and an-annular.  Then for any triangulation $\T$ of $M$ there is an algorithm to decide if there is a closed normal surface that is isotopic into $\bdy M$ but is not normally isotopic into $\bdy M$. Furthermore, if there is one the algorithm will construct one. \end{thm}

\begin{proof}  If $S$ and $S'$ are disjoint normal surfaces embedded in $M$ and both are isotopic into $\bdy M$, we say $S'$ is \emph{larger than} $S$ if $S$ is contained in the product region between $S'$ and $\bdy M$. Being larger than is a partial order on closed normal surfaces embedded in $M$.  

Suppose there is a normal surface in $M$ that is isotopic into $\bdy M$ but is \underline{not} normally isotopic   into $\bdy M$.  By Kneser's Finiteness Theorem \cite{kne} there are maximal (relative to the preceding partial order) such surfaces.  Suppose $S$ is a maximal normal surface that is isotopic into $\bdy M$ but not normally isotopic into 
$\bdy M$. We claim $S$ is a fundamental surface.

Suppose $S$ is not fundamental. Then $S = X + Y$ is a nontrivial Haken sum.  Hence, there are exchange annuli between $X$ and $Y$.   Suppose $A$ is an exchange annulus. Then $A$ is a 0-weight annulus meeting $S$ only in its boundary.  There are two possibilities: either $A$ is in not in the product region between $S$ and $\bdy M$ or $A$ is in the product region between $S$ and $\bdy M$.  Since $S$ is isotopic into $\bdy M$ and $M$ is an-annular, then for either possibility, $A$ is isotopic into $S$.   

Let $N=N(S\cup A)$ be a small regular neighborhood of $S\cup A$, then  $N$ has three boundary components; one is a torus bounding a solid torus, which is a product between $A$ and an annulus $A'$ in $S$, another is surface normally isotopic to $S$, and the third is a surface isotopic to $S$ but possibly not normal and even if normal is not  normally isotopic to $S$.  The complex $S\cup A$ is a barrier (see \cite{jac-rub-0eff}) and thus each boundary component of $N$ can be normalized in the closure of the component of its complement not meeting $S\cup A$. 

Suppose $A$ is not in the closure of the product region between $S$ and $\bdy M$.  Then
the component of $\partial N$ isotopic to $S$ can be normalized missing $S ? A$ to a normal surface $S?$.  Since $M$ is irreducible and $\bdy$--irreducible, $S'$ is isotopic to $S$ and therefore, isotopic into $\bdy M$.  Moreover $SÕ$ is not normally isotopic to $S$ or into $\bdy M$ due to the annulus $A$.  But $S'$ is larger than $S$, which contradicts $S$ being maximal.

Suppose $A$ is in the closure of the product region between $S$ and  $\bdy M$.  Then $A$ co-bounds a solid torus which is a product between $A$ and an annulus $A'$ in $S$.  We observe that $X \neq Y$, for if this were not the case, then $X$ (and $Y$) would be one-sided and $M$ would be a twisted I-bundle, contradicting $M$ being an-annular.  Hence, there must be a trace curve in $A'$. Suppose we have selected $A'$ in this situation so that it has a minimal number of trace curves.  Since there is a trace curve in $A'$, there is another exchange annulus $A_1$ for $S$ meeting $A'$ in at least one of its boundary components.  If $A_1$ is 
not in the closure of the product region between $S$ and $\bdy M$, then the preceding argument gives a contradiction to our selection of $S$.  So we may assume $A_1$  is, like $A$,  in the closure of the product region between $S$ and  $\bdy M$ and therefore in the solid torus co-bounded by  $A$ and $A'$.  It follows that $A_1$ 
co-bounds a solid torus which is a product between $A_1$ and an annulus $A_1'$ in $A'\subset S$.  However, then $A_1'$  has fewer trace curves than $A'$ contradicting our choice of the exchange annulus $A$. 

So, there is a closed normal surface that is isotopic into $\bdy M$ and not normally isotopic into $\bdy M$ if and only if there is such a surface among the fundamental surfaces for the triangulation $\T$ of $M$.  By \cite{jac-tol} given any normal surface we can determine if it is isotopic into $\bdy M$ and it is straight forward to recognize if it is normally isotopic into $\bdy M$. It follows if such a surface exists, we can construct one.  \end{proof}

The argument carries over to an analogous result in the case of an ideal triangulation.

\begin{cor}\label{determine-vertex-linking} Suppose $M$ is a compact, orientable 3--manifold with boundary that is irreducible,  $\bdy$--irreducible, and an-annular.  Then for any ideal triangulation $\T^*$ of $\open{M}$ and any ideal vertex $v^*$ of $\T^*$, there is an algorithm to decide if there is a closed normal surface that is isotopic into the vertex-linking surface of $v^*$ but is not normally isotopic into the vertex-linking surface of $v^*$. Furthermore, if there is one, the algorithm will construct one. \end{cor}

\section{Basics of crushing and inflating triangulations}
\subsection{Crushing triangulations along normal surfaces}  In \cite{jac-rub-0eff} we introduced the procedure of ``crushing a
triangulation along a normal surface."  Details may be reviewed there, as well as in  \cite{jac-rub-inflate}, where the details apply more directly to our situation in this work.

Suppose
$\T$ is a triangulation of the compact $3$--manifold $M$ or an ideal
triangulation of the interior of $M$. Suppose $S$ is a closed normal
surface in $M$, $X$ is the closure of a component of the complement
of $S$, and $X$ does not contain any of the vertices of $\T$. Since
$X$ does not contain any of the vertices of $\T$, the triangulation
$\T$ induces a particularly nice cell-decomposition on $X$, say $\mathcal{C}_X$, 
consisting of {\it truncated-tetrahedra, truncated-prisms, triangular
product blocks}, and {\it quadrilateral product blocks}.  See Figure
\ref{f-cell-decomp}. 

\begin{figure}[htbp]
            \psfrag{X}{$X$}
            \psfrag{s}{\small tetrahedron}
            \psfrag{f}{\small face}

             \psfrag{c}{{\tiny crush}}
            \psfrag{e}{\small edge}
            \psfrag{t}{\begin{tabular}{c}
          {\small truncated-tetrahedron}\\
            \end{tabular}}
            \psfrag{p}{\begin{tabular}{c}
            {\small truncated-prism}\\
            \end{tabular}}
            \psfrag{q}{\begin{tabular}{c}
          {\small triangular}\\
          {\small product block}\\
            \end{tabular}}
            \psfrag{r}{\begin{tabular}{c}
          {\small quadrilateral}\\
          {\small product block}\\
            \end{tabular}}

        \vspace{0 in}
        \begin{center}
        \epsfxsize=3.5 in
        \epsfbox{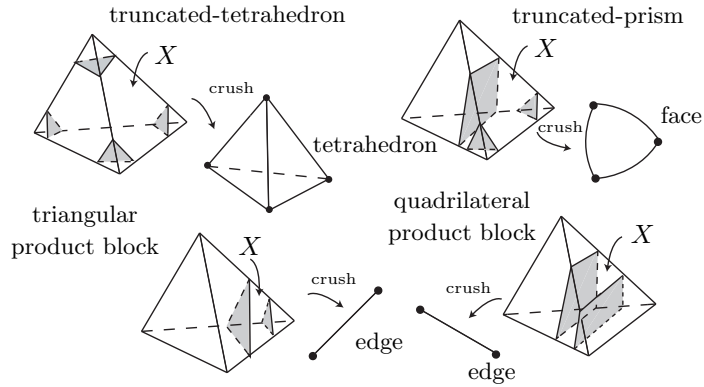}
        \caption{Cells in induced cell-decomposition $\mathcal{C}_X$ of $X$ and their crushing
         to tetrahedra, faces, and edges in an ideal triangulation of $\open{X}$.}
        \label{f-cell-decomp}
        \end{center}

\end{figure}

The boundary of each $3$--cell in $\C_X$ has an induced cell
decomposition in which some of the cells are in $S$ and some are
not. The edges and faces in the decomposition $\C_X$ are called {\it
horizontal} if their interiors are in $S$ and {\it vertical} if
their interiors are not in $S$. The quadrilateral vertical
$2$--cells are called {\it trapezoids}; there are two in a
truncated-prism, three in a triangular block, and four in a
quadrilateral block. The non-trapezoidal vertical $2$--cells are in
truncated-prisms and truncated-tetrahedra and are hexagons.

We define $\mathbb{P}(\C_X)$ as the union, $\mathbb{P}(\C_X)=$ $ \{$vertical edges of $\C_X  \}$ $ \cup \{$trapezoids$\}\cup\{$triangular
blocks$\}\cup\{$quadrilateral blocks$\}$.  Each component of $\mathbb{P}(\C_X)$ is an $I$--bundle.  Suppose each component of $\mathbb{P}(\C_X)$ is a product $I$ bundle. Then a component of $\mathbb{P}(\C_X)$ is a product $\mathbb{P}_i = K_i\times I$,  where $K_i^\ve =K_i\times \ve, \ve =0, 1$ and $K_i\times 0$ and $K_i\times 1$ are isomorphic  subcomplexes in the induced normal cell decomposition on $S$, $i=1,2,\ldots,k$, $k$ being the number of components of $\mathbb{P}(\C_X)$.  In this situation, we call $\mathbb{P}(\C_X)$ the {\it combinatorial product for $\C_X$}.  If $\mathbb{P}(\C_X) \not= X$ and each $K_i$ is a simply
connected planar complex (hence, it is cell-like), we say $\bbb{P}(\C_X)$ is a {\it trivial combinatorial product}.  In applications, we do not always have things so nice and we need to modify $\mathbb{P}(\C_X)$ to an {\it induced product region for $X$}, denoted $\mathbb{P}(X)$.

Now, consider the truncated-prisms in $\C_X$. Each truncated-prism has
two hexagonal faces. In $\C_X$, these hexagonal faces are identified
via the face identifications of the given triangulation $\T$ to a
hexagonal face of a truncated-tetrahedron or to a hexagonal face of
truncated-prism. If we follow a sequence of such identifications
through hexagonal faces of truncated-prisms, we trace out a
well-defined arc that terminates at an identification with a
hexagonal face of a truncated-tetrahedron or possibly does not
terminate but forms a complete cycle through hexagonal faces of
truncated-prisms. We call a collection of
truncated-prisms identified in this way a {\it chain}. If a chain
ends in a truncated-tetrahedra, we say the chain {\it terminates};
otherwise, we call the chain a {\it cycle of truncated-prisms}.

Just as in \cite{jac-rub-inflate}, under appropriate conditions, we can construct an ideal triangulation of $\open{X}$ using a controlled crushing of the cells of $\C_X$.   In particular, to obtain the desired  ideal triangulation of $\open{X}$ it is sufficient that  
$X\ne \bbb{P}(\C_X)$ or in the more general case $X\ne \bbb{P}(X)$ (there are not too many product blocks) and there are no
cycles of truncated-prisms (there are not too many truncated-prisms).  As a result of the crushing,  each component of $S$ is
crushed to a point (distinct points for distinct components), all designated 
products are crushed to arcs and, in particular, the products $K_i\times I$ are crushed to arcs (edges) so that if
$K_i\times I$ is crushed to the edge $e_i$, then the crushing
projection coincides with the projection of $K_i\times I$ onto the
$I$ factor.  Vertical edges, trapezoids, and product
blocks in $\C_X$ are identified to edges in the ideal triangulation. Truncated-prisms becomes faces
and truncated-tetrahedra become  tetrahedra.  Consult \cite{jac-rub-0eff} and see Figure \ref{f-cell-decomp}.

The crushing is particularly nice in the case that $\bbb{P}(\C_X)$ is a trivial combinatorial product, $X\ne \bbb{P}(\C_X)$, and there are no cycles of truncated prisms.  In this case, suppose $\{\overline{\Delta}_1,\ldots,\overline{\Delta}_n\}$ denotes the
collection of truncated-tetrahedra in $\C_X$.  Each
truncated-tetrahedron in $\C_X$ has its triangular faces in $S$. If we
crush each such triangular face of a truncated-tetrahedron to a
point (for the moment, distinct points for each triangular face), we
get a tetrahedron. We use the notation $\td{\Delta}_i^*$ for the
tetrahedron coming from the truncated-tetrahedron
$\overline{\Delta}_i$ after identifying  the triangular faces of
$\overline{\Delta}_i$ to points.  Also as a consequence of this crushing of $S$, if $\overline{\sigma}_i$ is a
hexagonal face in $\overline{\Delta}_i$, then $\overline{\sigma}_i$
is identified to a triangular face, say $\td{\sigma}_i^*$, of
$\td{\Delta}_i^*$.

Let $\mathbf{\bf{\td{\Delta}^*}} =
\{\td{\Delta}_1^*,\ldots,\td{\Delta}_n^*\}$ be the tetrahedra
obtained from the collection of truncated-tetrahedra
$\{\overline{\Delta}_1,\ldots,\overline{\Delta}_n\}$ following the
crushing of the normal triangles in the surface $S$ to points. It
follows that there is a family $\mathbf{\Phi}^*$ of face-pairings
induced on the collection of tetrahedra
$\mathbf{\bf{\td{\Delta}^*}}$ by the face-pairings of $\C_X$ (coming
from the face-pairings of $\T$) as follows:
\begin{itemize}\item[-] if the face $\overline{\sigma}_i$ of
$\overline{\Delta}_i$ is paired with the face $\overline{\sigma}_j$
of $\overline{\Delta}_j$, then this pairing induces the pairing of
the face $\td{\sigma}_i^*$ of $\td{\Delta}_i^*$  with the face
$\td{\sigma}_j^*$ of $\td{\Delta}_j^*$ ;\item [-] if the face
$\overline{\sigma}_i$ of $\overline{\Delta}_i$ is paired with a face
of a truncated-prism in a chain of truncated-prisms and the face
$\overline{\sigma}_j$ of the truncated-tetrahedron
$\overline{\Delta}_j$ is also paired with a face of this chain of
truncated-prisms, then the face $\td{\sigma}_i^*$ of
$\td{\Delta}_i^*$ has an induced pairing with the face
$\td{\sigma}_j^*$ of $\td{\Delta}_j^*$ through the chain of
truncated-prisms.\end{itemize}

Hence, we get a $3$--complex
$\boldsymbol{\td{\Delta}}^*/\boldsymbol{\td{\Phi}}^*$, which is a
$3$--manifold except, possibly, at its vertices. We will denote the
associated ideal triangulation by $\T^*$. We call $\T^*$ the ideal
triangulation obtained by {\it crushing the triangulation $\T$ along
$S$}. We denote the image of a tetrahedron $\td{\Delta}^*_i$ by
$\Delta^*_i$ and, as above, call  $\td{\Delta}^*_i$ the lift of
$\Delta^*_i$.

We have the following  version of the Fundamental Theorem for Crushing Triangulations along a  Normal Surface.  A more general version and its proof appear in \cite{jac-rub-0eff}.  

\begin{thm}\label{combinatorial-crush} Suppose $\T$ is a triangulation of a compact, orientable $3$--manifold
or an ideal triangulation of the interior of a compact, orientable
$3$--manifold $M$. Suppose $S$ is a closed normal surface embedded in
$M$, $X$ is the closure of a component of the complement of $S$,
and $X$ does not contain any vertices of $\T$. If
\begin{enumerate}
\item[i)] $X\ne \bbb{P}(\C_X)$ \item[ii)] $\bbb{P}(\C_X)$ is a trivial
product region for $X$,  and \item[iii)] there are no cycles of
truncated prisms in $X$,  \end{enumerate}
then the triangulation $\T$ can be crushed along $S$ and the ideal triangulation $\T^*$ obtained by crushing $\T$ along $S$ is
an ideal triangulation of $\open{X}$.
\end{thm}

In this situation, we say the triangulation $\T$ admits a \emph{combinatorial crushing along $S$}. Notice that in the case of a combinatorial crushing, the tetrahedra in the ideal triangulation $\T^*$ are in one-one correspondence with the truncated tetrahedra in the cell decomposition $\C_X$ of $X$.  The latter come from truncating a sub collection of the tetrahedra of $\T$ and can be thought of as actually being a sub collection of the tetrahedra of $\T$; the face identifications for $\T^*$ are induced by the face identifications of $\T$.

\subsection{Inflating ideal triangulations}

A triangulation $\T$ of the compact $3$--manifold $M$ is said to be
a {\it minimal-vertex triangulation} if for any other triangulation
$\T_1$ of $M$ the number of vertices of $\T$ is no more than the
number of vertices of $\T_1$, $\abs{\T^{(0)}}\le\abs{\T_1^{(0)}}$.
If $M$ is closed, then $M$ has a one-vertex triangulation; hence, a minimal-vertex triangulation of $M$ is a
one-vertex triangulation \cite{jac-rub-layered}. If $M$ is a
compact $3$--manifold with boundary, no component of which is a
$2$--sphere, then $M$ has a triangulation with all of its
vertices in the boundary and then just one vertex in each boundary
component; hence, for such a manifold a minimal-vertex triangulation has all the vertices in the boundary and then just one vertex in each boundary component \cite{jac-rub-0eff}. These are the triangulations we are interested in and rather than write all of this out, we just say minimal-vertex triangulation.

 If $M$ is a compact 3--manifold with boundary and $\T$ is a triangulation of $M$ with normal boundary, then if $\T$ admits a crushing along the normal boundary, we say $\T$ can be \emph{crushed along $\bdy M$}. 
 
 If $S$ is a triangulated surface, we say that a subcomplex $\xi$ in the 1--skeleton of the triangulation of $S$ is a \emph{frame} in $S$, if $\xi$ is a spine for $S$ (its complement is a connected open cell in $S$)  and is a minimum among spines, with respect to set inclusion.  In any triangulation of $S$ there are many choices for a frame.  See Figure \ref{f-frames} for examples of frames in a torus.   A vertex in a frame is called a \emph{branch} or \emph{branch point}  if it has index greater than two.  The closure of a component of a frame minus its branch points is called a \emph{branch}.  For the examples in Figure \ref{f-frames}, that on the left has one branch point of index 4 and two branches while that on the right has two branch points, each of index 3, and three branches.

\begin{figure}[htbp]

            \psfrag{L}{$\xi$}
        \vspace{0 in}
        \begin{center}
\epsfxsize =2.75 in \epsfbox{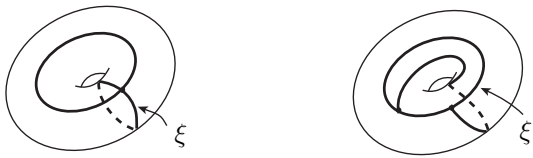} \caption{There are 
only two possible topological types for frames in a triangulation of the torus.} \label{f-frames}
\end{center}
\end{figure}

 \begin{defn} If $\T^*$ is an ideal triangulation of $\open{X}$, the interior of the compact 3--manifold $X$, an \emph{inflation of $\T^*$} is a minimal-vertex triangulation $\T$ of $X$ with normal boundary that admits a \underline{combinatorial crushing} along $\bdy X$ for which the ideal triangulation obtained by crushing $\T$ along $\bdy X$ is the ideal triangulation $\T^*$ of $\open{X}$. \end{defn}

A construction for inflations of ideal triangulations of 3--manifolds is developed  in \cite{jac-rub-inflate}.  We discuss the construction here but reference the reader to \cite{jac-rub-inflate} for complete details. For a given ideal triangulation of the interior of a compact 3--manifold, there is not a unique inflation; however, all inflations of a given ideal triangulation share many common properties, some of which play a crucial role in this work.

Our construction begins with the choice of a ``frame" in the 1--skeleton of the induced triangulation of each vertex-linking surface of $\T^*$.  If $v^*$ is an ideal vertex of $\T^*$, we will use the notation $S_{v^*}$ for the vertex-linking surface of $v^*$ and $\xi$ for a frame in $S_{v^*}$. If there are a number of ideal vertices,  then for an ideal vertex $v_i^*$ we use $S_{v_i^*}$ for the vertex-linking surface and $\xi_i$ for a frame in $S_{v_i^*}$. We let $\Lambda = \xi_1\cup\xi_2\cup\cdots\cup\xi_k$ denote the union of the frames from all the vertex-linking surfaces.

 An inflation of an ideal triangulation $\T^*$ of $\open{X}$ includes all the tetrahedra of $\T^*$ and then, guided by the frame $\Lambda$, new tetrahedra are added to the tetrahedra of $\T^*$ and new  face identifications  are determined (discarding some of the face identifications of $\T^*$, using some of the face identifications of $\T^*$, and adding some new face identifications)  to arrive at a minimal-vertex triangulation $\T_\Lambda$ of $X$.  The triangulation $\T_\Lambda$ will have normal boundary that admits a combinatorial crushing of $\T_\Lambda$ along  $\bdy X$, crushing $\T_\Lambda$ back to the ideal triangulation  $\T^*$.  Figure \ref{f-inflate-scheme-lite-ann} provides a schematic for going between an ideal triangulation $\T^*$ of $\open{X}$ and a minimal-vertex triangulation $\T_\Lambda$ of $X$.

 \begin{figure}[htbp]

            \psfrag{X}{$(X,\T_\Lambda)$}\psfrag{Y}{$(\open{X},\T^*)$}
            \psfrag{1}{\tiny {$v_1^*$}} \psfrag{2}{\tiny {$v_2^*$}} \psfrag{3}{\tiny {$v_3^*$}}
            \psfrag{a}{\tiny {$\xi_1$}} \psfrag{b}{\tiny {$\xi_2$}} \psfrag{c}{\tiny {$\xi_3$}}
            \psfrag{x}{\tiny {$B_{\tiny {v_1}}$}} \psfrag{y}{\tiny {$B_{\tiny {v_2}}$}}\psfrag{z}{\tiny {$B_{\tiny {v_3}}$}}
            \psfrag{A}{\scriptsize {$S_{\tiny {v_1^*}}$}}\psfrag{B}{\scriptsize {$S_{\tiny {v_2^*}}$}}
            \psfrag{C}{\scriptsize {$S_{\tiny {v_3^*}}$}}
            \psfrag{u}{\scriptsize {$S_{\tiny {v_1}}$}}\psfrag{v}{\scriptsize {$S_{\tiny {v_2}}$}}
            \psfrag{w}{\scriptsize {$S_{\tiny {v_3}}$}}
            \psfrag{m}{\scriptsize
            {\bf Inflate}} \psfrag{n}{\scriptsize
            {\bf Crush}}
            \psfrag{L}{$\Lambda =\xi_1\cup\xi_2\cup\xi_3$}

        \vspace{0 in}
        \begin{center}
        \epsfxsize=4.5 in
        \epsfbox{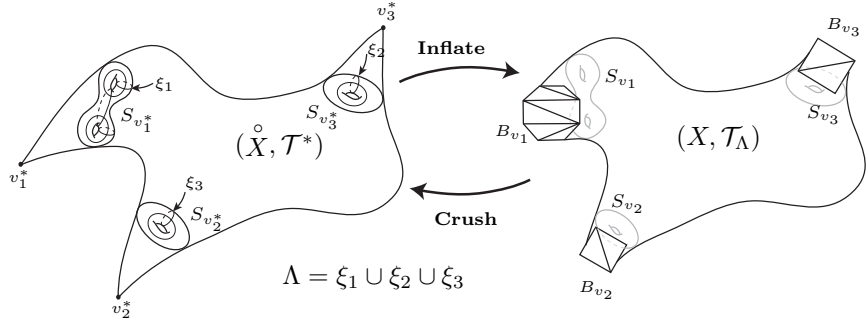}\caption{ Schematic of an inflation of an ideal triangulation using the collection of frames in $\Lambda$}
        \label{f-inflate-scheme-lite-ann}
\end{center}
\end{figure}

 An ideal vertex $v^*$ in $\T^*$ inflates to a minimal (one-vertex) triangulation of a component $B_v$ of $\bdy X$, which is induced by $\T_\Lambda$.   The vertex-linking surface $S_{v^*}$ about the ideal vertex $v^*$  inflates to a normal boundary $S_v$ in the triangulation $\T_\Lambda$, which is boundary-linking $B_v$.

\subsection{Closed normal surfaces} In \cite{jac-rub-inflate} a
one-one correspondence is given between the closed normal surfaces
in an ideal triangulation $\T^*$ and the closed normal surfaces in
any inflation $\T$  of $\T^*$. A special case of this relationship
is a key ingredient for the work in this paper and relates  boundary
parallel normal surfaces in an inflation of an ideal triangulation
with normal surfaces parallel to vertex-linking surfaces in the
ideal triangulation. We provide the details for this special case in
this section.

\begin{lem} Suppose $M$ is a compact $3$--manifold with nonempty
boundary and $\T$ is a triangulation of $M$ with normal boundary. An
embedded normal surface in $\T$ that contains all the quad types of
a boundary-linking surface has that boundary-linking surface as a
component.\end{lem}

\begin{proof} Suppose $S$ is an embedded normal surface in $\T$, $\td{B}$
is a boundary-linking surface, and all quad types of $\td{B}$ are
represented as quad types in $S$. Then $S$ and $\td{B}$ are
contained in the carrier of $S$, a face of compatible (no two
distinct quad types in the same tetrahedron) normal solutions in the
solution space of embedded normal surfaces. It may be the case that
 $\td{B}$ is in a proper
face. Since, $\td{B}$ has no more quad types that $S$, it follows
that there is a normal surface $R$ and positive integers $k,n,$ and
$m$ so that $kS = nR + m\td{B}$. However, we can move $\td{B}$ by a
normal isotopy so that it does not meet $R$. Hence, we have $\td{B}$
a component of $kS$ and therefore a component of $S$.\end{proof}

\begin{lem}\label{crush normal to normal} Suppose $M$ is a compact $3$--manifold with nonempty boundary, no component of which is a 2--sphere. Suppose
 $\T^*$ is an ideal triangulation of
$\open{M}$, and $\T$ is an inflation of $\T^*$.  The combinatorial crushing map determined by crushing $\T$ along $\bdy M$ takes a closed normal
surface $S$ in $\T$  to a closed normal surface $S^*$ in
$\T^*$; furthermore, $S$ and $S^*$ are homeomorphic.
\end{lem}

\begin{proof}  Let $X$ denote the component of the complement of the
boundary-linking surfaces that does not meet $\bdy M$. Then $X$
contains none of the vertices of $\T$ and has a nice
cell-decomposition $\mathcal{C}$; furthermore, this
cell-decomposition combinatorially crushes along the boundary-linking surfaces to
the ideal triangulation $\T^*$.

Let $S$ be a closed normal surface in $\T$. The surface $S$ has an induced cell-decomposition from $\T$
consisting of normal quadrilaterals and normal triangles. Since $S$
is a closed normal surface, we may assume $S$ does not meet any of
the boundary-linking surfaces along which we are crushing, and thus
$S\subset X$.

If a normal quad or normal triangle of $S$ is in a
truncated-tetrahedron in $\C$, then upon crushing, the
truncated-tetrahedron is taken to a tetrahedron of $\T^*$ and the
normal cells of $S$ in the truncated-tetrahedron are carried
isomorphically onto normal cells in $\T^*$ (see Figure
\ref{f-crush-normal} A). If a normal quad or normal triangle of $S$
is in a truncated-prism of $\C$, then the truncated-prism is crushed
to a face in $\T^*$ and the normal cells of $S$ are crushed to
normal arcs in that face. The normal arcs in the hexagonal faces of
the truncated-prisms correspond to where $S$ meets these hexagonal faces and are matched under the crushing map from the various truncated prisms in a chain of truncated prisms. Arcs in
the trapezoidal faces of the truncated-prism crush to points in the edges of the face in which the truncated prism crushes (see
Figure \ref{f-crush-normal} B). Finally, the normal cells of $S$ in
the product blocks of $\C$ are ``horizontal" triangles in the
triangular product blocks and ``horizontal" quadrilaterals in the
quadrilateral blocks and, hence, each is crushed to a single point
in an edge of $\T^*$ (see
Figure \ref{f-crush-normal} C). The crushing in the trapezoidal faces of the
truncated-prisms and the product blocks are consistent. It follows
that the image of $S$ is formed from the collection of normal
triangles and normal quadrilaterals of $S$ that are in the
truncated-tetrahedron of $\C$ by identifications along their edges
and gives a normal surface $S^*$ in $\T^*$.

 \begin{figure}[htbp]

        \vspace{0 in}
        \begin{center}
        \epsfxsize=3.5 in
        \epsfbox{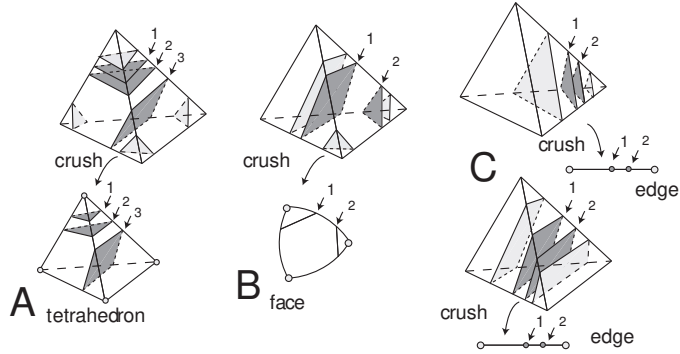}\caption{A.Normal disks in truncated tetrahedra go to normal disks. B.Normal disks in truncated prisms go to normal arcs. C.Normal disks in product blocks go to points. }
        \label{f-crush-normal}
\end{center}
\end{figure}

To see that $S$ and $S^*$ are homeomorphic, we observe that the
inverse image of a point in the interior of a normal quad or normal
triangle in $S^*$ is a point in the interior of a normal quad or
normal triangle in $S$. The inverse image of a point in an edge of
$S^*$ is either a point in an edge of $S$ or a sequence of arcs in normal quads or normal triangles of $S$; there are no
cycles of truncated prisms and so no cycles of cells of $S$ in
truncated-prisms. The inverse image of a vertex of $S^*$ is a horizontal cross section $K_i\times {t}$ in one of the component product pieces $\mathbb{P}_i = K_i\times I$ of the combinatorial product $\mathbb{P}(\C_X)$. Hence, $K_i$ is a contractible planar complex.  Thus for each point of $\S^*$ its inverse image in $S$ is a contractible planar complex and so the combinatorial crushing map gives a cell-like map from $S$ to $S^*$ and by a 2-dimensional versions of \cite{arm, sie} it follows that $S$ and $S^*$ are homeomorphic.
\end{proof}

\begin{thm} \label{bijection-ideal-inflate} Suppose $M$ is a compact 3--manifold with nonempty boundary no component of which is a 2--sphere. Suppose
 $\T^*$ is an ideal triangulation of
$\open{M}$, and $\T$ is an inflation of $\T^*$.  The combinatorial crushing map determined by crushing $\T$ along $\bdy M$ induces a bijection between the closed normal
surfaces  in $\T$  and the closed normal surface in
$\T^*$; furthermore, corresponding surfaces are homeomorphic.\end{thm}
\begin{proof} By Lemma \ref{crush normal to normal} we only need to show that the combinatorial crushing induces a bijection between the closed normal surfaces of $\T$ and those of $\T^*$.  

First we shall show that the correspondence is injective. Suppose $S_1$ and $S_2$ are distinct closed normal surfaces in $\T$. Since both are closed, we may assume that (up to normal isotopy) they do not meet any boundary-linking surface in $\T$.  Let $X$ denote the component of the complement of the boundary-linking normal surfaces in $\T$, which does not meet $\bdy M$ and let $\C_X$ denote the nice cell decomposition on $X$ induced by $\T$.   Then $S_1$ and $S_2$ are distinct normal surfaces in $\C_X$  and hence, have distinct normal coordinates. By Lemma \ref{crush normal to normal} the combinatorial crushing of $\T$ to $\T^*$ takes $S_1$ and $S_2$ to closed normal surfaces $S_1^*$ and $S_2^*$ in $\T^*$, respectively.   We will show that $S_1^*$ and $ S_2^*$ have distinct normal coordinates.

If $S_1$ and $S_2$ have distinct sets of normal disks in a truncated tetrahedron of $\C_X$, then $S_1^*$ and $S_2^*$ have distinct normal disks in a tetrahedron of $\T^*$ and hence, $S_1^* \ne S_2^*$. 

If $S_1$ and $S_2$ have distinct normal disks in a truncated prism, say $\pi$, then they have distinct sets of normal arcs in a hexagonal face of $\pi$, which extend to distinct sets of normal arcs on all the hexagonal faces of the truncated prisms in the chain of truncated prisms containing $\pi$, which leads to a distinct set of normal disks in a truncated tetrahedron in which the chain terminates. This again  gives that $S_1^* \ne S_2^*$.  Finally if $S_1$ and $S_2$ have a distinct number of quads or triangles in a quadrilateral or triangular block, respectively, then $S_1$ and $S_2$ meet an entire product component, $K_i\times I$ in a distinct number of horizontal slices. The vertical frontier of a product $K_i\times I$ is made up of trapezoidal faces which are paired with trapezoidal faces of truncated prisms. Thus we have that $S_1$ and $S_2$ must meet a truncated prism in distinct normal disks. From the previous consideration, we have that $S_1^*$ and $S_2^*$ are distinct.  So, the correspondence is injective.

Now, we must show the correspondence is surjective. Suppose $S^*$ is a closed normal surface in $\T^*$.  First, we consider how, $S^*$ meets a tetrahedron of $\T^*$. Each tetrahedron of $\T^*$ is the image of a single truncated tetrahedron of $\C_X$ under the crushing map; hence, there is a unique choice of normal cells in these truncated tetrahedra of $\C_X$ (tetrahedra of $\T$) mapping to the normal cells of $S^*$. If $\alpha^*$ is a face of a tetrahedron of $\T^*$ and $\alpha^*$ meets $S^*$, then the inverse image of $\alpha^*$ is either a single face between two truncated tetrahedra in $\C_X$ or is the image of a chain of truncated prisms in $\C_X$ between two truncated tetrahedra in $\C_X$. If the inverse image of $\alpha^*$ is a single face matching two truncated tetrahedra, then there are well determined normal cells in each of these truncated tetrahedra determined by the normal cells in $S^*$. If there is a chain of truncated prisms determined by $\alpha^*$, then of the three possible families of normal arcs in $\alpha^*$, only one of the families determines quadrilaterals in any one of the truncated prisms in the chain determined by $\alpha$.  This again determines a unique way to fill in normal disks extending the normal disks in the truncated tetrahedra. Finally, for each product $K_i\times I$ there is a unique number of horizontal slices determined to complete a normal surface $S$ in $\T$ that crushes to $S^*$.
\end{proof}

Recall that any  ideal triangulation of the interior of a compact 3--manifold M with boundary, no component of which is a 2--sphere, has numerous inflations. By the previous theorem, all of these inflations have isomorphic sets of closed normal surfaces. Here we use isomorphic to mean a bijection between the sets of normal surfaces where corresponding surfaces are homeomorphic. 

\begin{cor}\label{vertex-link is bdry-link} Suppose $M\ne \mathbb{B}^3$ is a compact, irreducible and
$\bdy$-irreducible $3$--manifold with nonempty boundary. Suppose
 $\T^*$ is an ideal triangulation of
$\open{M}$, and $\T$ is an inflation of $\T^*$.  There is a closed normal
surface in $\T$  isotopic into  $\bdy M$ but not normally isotopic into $\bdy M$ if and only if there is
 a closed normal surface in $\T^*$ that
is isotopic into a vertex-linking surface but is not normally isotopic into a  vertex-linking surface.\end{cor}

\begin{proof}  Suppose $S$ and $S^*$ are closed normal surfaces in $\T$ and $\T^*$, respectively, that correspond under the combinatorial crushing map taking $\T$ to $\T^*$. Then the closure of the components of the complement of $S$ have a correspondence under the combinatorial crushing map and corresponding components are homeomorphic.  Also, we have that the correspondence under the combinatorial crushing map  takes boundary-linking normal surfaces in $\T$ to vertex-linking normal surfaces in $\T^*$. Hence, if $S$ is isotopic into $\bdy M$, $S$ is isotopic to a boundary-linking surface and so, $S^*$ is isotopic to a vertex-linking surface in $\T^*$. The converse also follows.
\end{proof}

\section{annular-efficient triangulations}

If $M$ is a $3$--manifold and $\T$ is a triangulation, we say the
triangulation $\T$ is {\it $0$--efficient} if
\begin{enumerate}\item[(i)] $M$ is closed and the only normal
$2$--spheres are vertex-linking; or\item[(ii)] $\bdy M\ne\emptyset$
and the only normal disks  are vertex-linking.\end{enumerate}

It is shown in  Proposition 5.1 and Proposition 5.15 of
\cite{jac-rub-0eff} that if $\T$ is a $0$--efficient triangulation
of the compact $3$--manifold $M$, then if
\begin{enumerate}\item[(i)] $M$ is closed,  then $M \ne \rppp$, is irreducible,
 and $\T$ has only one vertex, or $M=S^3$ and $\T$ has precisely two
 vertices.
\item[(ii)] $\bdy M\ne\emptyset$, then $M$ is irreducible,
$\bdy$--irreducible, there are no normal $2$--spheres, all the
vertices are in $\bdy M$, and there is precisely one vertex in each
boundary component, or $M = \mathbb{B}^3$.\end{enumerate}

\begin{figure}[htbp] \psfrag{B}{\tiny
{$B\subset\bdy M$}}\psfrag{N}{{$e \in \bdy M$}}\psfrag{M}{{$e \in
\open{M}$}}
 \psfrag{b}{\tiny{$B'\subset\bdy M$}}\psfrag{e}{\footnotesize{e}}
 \psfrag{V}{\tiny{$v$}}\psfrag{v}{\tiny
 {$v'$}}\psfrag{D}{\footnotesize{$D$}}\psfrag{d}{\footnotesize{$D'$}}\psfrag{A}{\footnotesize{cycle
 of quads}}

        \begin{center}
\epsfxsize =3.5 in \epsfbox{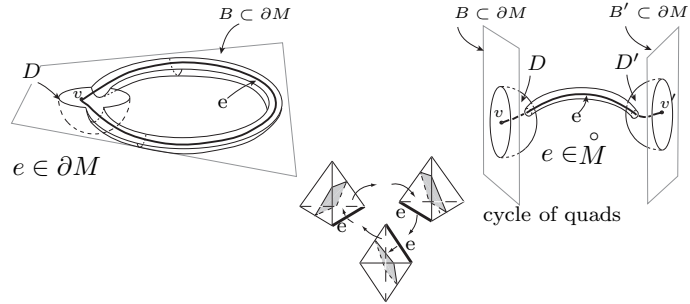} \caption{Thin
edge-linking annuli. On the left is a thin edge-linking annulus
about the edge $e$ in the boundary. On the right, the edge $e$ is in
the interior of the manifold.} \label{fr-thin-annulus.eps}
\end{center}
\end{figure}

If $\T$ is a triangulation of the $3$--manifold $M$, we say a normal
annulus in $M$ is {\it thin edge-linking} if and only if it is
normally isotopic to an arbitrarily small regular neighborhood of an
edge in the triangulation. In Figure \ref{fr-thin-annulus.eps}, we
give examples of thin edge-linking annuli; the figure on the left is
a thin edge-linking annulus for an edge $e$  in $\bdy M$ and the one
on the right is a thin edge-linking annulus for an edge $e$ in
$\open{M}$. Notice that a thin edge-linking annulus about the edge
$e$ is determined from one or two vertex-linking disks by removing
all normal triangles that meet $e$ and replacing them with normal
quads that do not meet $e$ but are in the tetrahedra containing $e$.
A necessary and sufficient condition for an edge $e$ in $\bdy M$ or
a properly embedded edge $e$ in $M$ to have a thin edge-linking
annulus about it is that no face in the triangulation has two edges
identified to $e$.

If $M$ is a compact $3$--manifold with nonempty boundary and $\T$ is
a triangulation of $M$, we say $\T$ is {\it annular-efficient} if
and only if $\T$ is $0$--efficient and the only normal, incompressible annuli are thin edge-linking annuli. David Bachman and
Saul Schleimer, who have independently studied annular-efficient triangulations, call a triangulation {\it $1/2$--efficient} if it is annular-efficient in our sense.

\begin{prop} \label{an-annular-properties} Suppose  $M\ne \mathbb{B}^3$ is a compact $3$--manifold with boundary and has an
annular-efficient triangulation. Then $M$ is irreducible,
$\bdy$--irreducible, and an-annular. Furthermore, there are no normal
 $2$--spheres, all the vertices are in $\bdy M$, the only normal disks are vertex-linking, and there
is precisely one vertex in each component of $\bdy M$.\end{prop}
\begin{proof} Since an annular-efficient triangulation is
$0$--efficient, it follows from Theorem 5.15 of \cite{jac-rub-0eff}
that $M$ is irreducible and $\bdy$--irreducible and there are no
normal $2$--spheres, all the vertices are in $\bdy M$, and there is
precisely one vertex in each boundary component of $M$. Hence, it remains to prove that $M$ is an-annular. If
there is a properly embedded, essential annulus in $M$, then for any
triangulation, there must be a normal, embedded, essential annulus
in $M$. In particular, this would need to be the case for the given
annular-efficient triangulation. However, a normal, embedded,
essential annulus can not be thin edge-linking, as a thin
edge-linking annulus is parallel into the boundary of the manifold
and, therefore, is not essential. Thus the assumption of an
embedded, essential annulus leads to a contradiction.\end{proof}

\begin{prop}\label{decide-annular-eff} Given a triangulation of a compact, orientable 3--manifold with nonempty boundary, no component of which is a 2--sphere, there is an algorithm to decide if the triangulation is annular--efficient. Furthermore, if the triangulation is not 0-efficient, the algorithm will construct a normal disk that is not vertex-linking; and if the triangulation is 0--efficient and is not annular-efficient,  the algorithm will construct an incompressible,  normal annulus that is not  thin edge-linking.\end{prop}

\begin{proof} By Proposition 5.19 of \cite{jac-rub-0eff}, it can be decided if the given triangulation is 0--efficient; and, if there is a normal disk that is not vertex-linking, the algorithm will construct one.  So, we may assume the only normal disks are vertex-linking; i.e., the triangulation is 0--efficient.

If there is an incompressible, normal annulus that is not thin edge-linking, then consider one,  say $A$, where the carrier of $A, \C(A)$, has minimal dimension.  If $\C(A)$ is not a vertex, then there are normal surfaces $X$ and $Y$ in proper faces of $\C(A)$ and positive integers $k, n$ and $m$ so that $kA = nX + mY$.  Since the triangulation is 0--efficient, the only positive Euler characteristic normal surfaces are vertex-linking normal disks; hence, neither $X$ nor $Y$ has a component with positive Euler characteristic.  It follows that the components of both $X$ and $Y$ have Euler characteristic zero.  Since $A$ has essential boundary, every component of $X$ and $Y$ with boundary has essential boundary. It follows that the components of $X$ and $Y$ with boundary are  either an  annulus or  a M\"obius band.  However, if a component of $X$ or $Y$ is a M\"obius band, then we would have a normal annulus that is not thin edge-linking and  carried by a proper face of $\C(A)$, which contradicts our choice of $A$.  So, any component of $X$ or $Y$ that has boundary is an annulus and, again by our choice of $A$, these annuli must be thin edge-linking.  Now, a thin edge-linking annulus can be normally isotoped to miss any closed normal surface. It follows that both $X$ and $Y$ must have components with boundary and as such both must have components that are thin edge-linking annuli. However, the Haken sum of two thin edge-linking annuli is either two thin edge-linking annuli (the two annuli are the same or have their boundaries in distinct boundaries of the 3--manifold), or has a component a vertex-linking disk.  Both possibilities lead to a contradiction that $A$ is connected and not a thin edge-linking annulus.  

It follows that if the triangulation is 0--efficient and there is a normal annulus that is not thin edge-linking, then there is one at a vertex of the projective solution space for the triangulation. Furthermore, we can recognize if a normal surface is a thin edge-linking annulus. \end{proof}

The following two results are from  \cite{jac-rub-0eff} and provide converses to Proposition 5.1 and Proposition 5.15 of that work.

\begin{thm}  If $M$ is a closed, orientable, irreducible
$3$--manifold distinct from $\rppp,$  then there is an algorithm that will modify any triangulation of $M$  to a $0$--efficient
triangulation.\end{thm}

\begin{thm}  If $M\ne \mathbb{B}^3$ is a compact, orientable, irreducible,
$\bdy$--irreducible $3$--manifold, with non-empty boundary, then there is an algorithm that will modify any triangulation of $M$  to a
$0$--efficient triangulation.\end{thm}

Theorem \ref{thm:annular-eff}  is a converse to Proposition \ref{an-annular-properties} and is our main theorem.   Bachman and Schleimer communicated to us \cite{bac-sch-per-comm} that they have an independent proof that a compact irreducible, $\bdy$--irreducible, an-annular 3-manifold admits an annular-efficient triangulation.  The proof given here is constructive and follows from the methods introduced in \cite{jac-rub-0eff}; we believe their methods may be different.  

\begin{thm}\label{thm:annular-eff} Suppose $M\ne \mathbb{B}^3$ is a compact,
irreducible, $\bdy$--irreducible,  an-annular $3$--manifold with
nonempty boundary. Then there is an algorithm that will modify any triangulation of $M$  to
an annular-efficient triangulation of $M$.
\end{thm}

Before giving the proof, we outline our approach and provide some
results needed in our proof.

\vspace{.125 in}\noindent (Outline of Proof). We are given a compact
$3$--manifold $M$ with boundary via a triangulation $\T$. We are
also given that $M$ is irreducible, $\bdy$--irreducible, and
an-annular. Note that given a $3$--manifold with boundary, algorithms
exist to determine if it is irreducible \cite{sch, jac-tol},
$\bdy$--irreducible \cite{haken-norflach, jac-tol, rubin-poly, tho},
or an-annular \cite{haken-homo, jac-sed-dehn}.  However, we assume in this work 
that we are given that the manifold is irreducible, $\bdy$--irreducible, and an-annular.

If there is a normal disk in $\T$ that is not vertex-linking or a
normal annulus with essential boundary in $\T$ that is not thin
edge-linking, then by a barrier surface argument, there is a closed
normal surface in $\T$ that is isotopic into $\bdy M$ but is not
\underline{normally} isotopic into $\bdy M$.  Hence, we can prove
Theorem \ref{thm:annular-eff} if we can modify the triangulation $\T$  so that the only normal surface isotopic into a component of $\bdy M$ is boundary-linking (in particular, $\T$ must have a normal boundary). To do this, we first
modify the given triangulation to an ideal triangulation of $\open{M}$,
the interior of $M$, and then modify this ideal triangulation to an
ideal triangulation of $\open{M}$ having the property that a normal
surface isotopic to a vertex-linking surface is also
\underline{normally} isotopic to that vertex-linking surface.  We
then rebuild a triangulation of $M$ by inflating this ideal
triangulation of $\open{M}$ and use Corollary \ref{vertex-link is bdry-link} to conclude that the inflation is
annular-efficient. 

Theorems 7.1 and 7.2 of \cite{jac-rub-0eff} establish, under our hypothesis, the  existence of a
0--efficient ideal triangulation of $\open{M}$.  Marc Lackenby proved in \cite{lack-taut} that with our hypotheses and the addition condition that every boundary component is an annulus,  $\open{M}$ admits a taut ideal triangulation, from which it follows that the triangulation also is 0--efficient.  While a taut ideal triangulation, which exists under additional hypothesis, implies that the ideal triangulation is  0--efficient, neither the results of \cite{jac-rub-0eff}  or \cite{lack-taut} give us what we need for our proof; and in neither of these referenced results do we have constructive proofs. Hence, we first establish a constructive proof that modifies the given triangulation of $M$ to an ideal triangulation of $\open{M}$.   Having constructed an ideal triangulation of $\open{M}$, we modify this ideal triangulation, if necessary,  to obtain an ideal triangulation in which the only normal surfaces isotopic to a vertex-linking surface are normally isotopic to it. This condition will enable us to construct an annular-efficient triangulation of the compact 3--manifold $M$ via an inflation. 

The next theorem requires a more general version of crushing a triangulation than that required earlier for a combinatorial crushing; however, this general version is precisely the version from Section 4 of \cite{jac-rub-0eff} and our theorem here is a constructive version of Theorem 7.1 of \cite{jac-rub-0eff}.
 
\begin{thm}\label{thm:construct-ideal} Suppose $M$
is a compact, irreducible, $\bdy$--irreducible, an-annular
$3$--manifold. Then for any triangulation $\T$ of $M$, there is an algorithm to modify the  triangulation $\T$ to an
ideal triangulation of $\open{M}$.\end{thm}

\begin{proof} $M$ is given by the triangulation $\T$. Let
$B_1,\ldots,B_n$ denote the components of $\bdy M$. 

If $v$ is a vertex of $\T$ in $\open{M}$, then there is an embedded
arc in the $1$--skeleton of $\T$ having $v$ as one end point and
meeting $\bdy M$ only in its other end point, a vertex of $\T$ in
$\bdy M$. 

 Put  an order on the vertices of $\T$ in $\open{M}$ and
 construct a finite number of pairwise disjoint trees,
$L_1,\ldots,L_K$ in the 1--skeleton of $\T$ so that for each $j,
1\le j\le K$, the tree $L_j$ meets $\bdy M$ in a single vertex of
$\T$ and every vertex of $\T$ in $\open{M}$ is in $L_j$ for some
$j$. For each $i, 1\le i\le n$, let $\hat{B}_i$ denote the boundary component $B_i$ of $M$ along with
all trees $L_{i_j}$ that meet $B_i$. Let $N$ denote a small regular neighborhood of $\cup_{i=1}^n \hat{B}_i$ and let $N_i$ be the component of $N$ containing $\hat{B}_i, 1\le i\le n$.  The frontier of $N$ is a
barrier surface for the component of the complement of $N$ that does
not meet any $\hat{B}_i$; and if $E_i$ is the frontier of the
component $N_i$ of $N$, then $E_i$ is isotopic into
the boundary component $B_i$.

The argument from here follows that in the proof of Theorem 7.1 of \cite{jac-rub-0eff} except here we want the argument to be constructive.  We shall indicate how the steps of that proof can be made constructive; however, we  encourage the 
reader interested in all the details to look at the presentation in
\cite{jac-rub-0eff}.

 Shrink each $E_i$ to a stable surface in the component of the complement of the frontier of $N$ not meeting $\cup\hat{B}_i$. This is constructive and we arrive at a
pairwise disjoint collection of normal surfaces (and possibly some
2--spheres, interior to tetrahedra, which may be discarded) so that for each $E_i$,  there is
precisely one normal surface that is isotopic to $E_i$ and therefore
is isotopic into the component $B_i$ of $\bdy M$. We continue to
call this, now normal, surface $E_i$ and denote the product region
determined by the isotopy of $E_i$ into $B_i$, by $P_i, 1\le i\le
n$. There is nothing to verify in this step as the conditions on $M$, leave no other possibilities.

Let $X$ denote the closure of the component of the complement of $\cup E_i$ that does not meet any $P_i;  X$ is
homeomorphic to $M$.  Furthermore, $X$ does not contain any vertices of
$\T$. Hence, we have that the triangulation $\T$ induces a nice
cell-decomposition $\C_X$ of $X$ and we can proceed to apply our methods to
crush the triangulation $\T$ along the normal surfaces
$E_1,\ldots, E_n$.

To do this we must verify that we have the sufficient conditions for crushing a triangulation along a normal surface.    In the proof of Theorem 7.1 of \cite{jac-rub-0eff},  we argued that we could assume the collection  $E_1\ldots, E_n$ as above satisfied a certain maximal condition.  Here we show that we can construct a collection along which we can crush. We may actually discover it before we get to a maximal collection in the sense of \cite{jac-rub-0eff}.

We begin by constructing the combinatorial product $\bbb{P}(C_X)$.  It follows immediately that $\bbb{P}(C_X) \ne X$; for if $\bbb{P}(C_X) = X$, then $M$ would be an I-bundle, which contradicts $M$ an-annular.

Next we have to show that we can get to a situation where we have a trivial induced product region.  

Recall that a component of the combinatorial product $\bbb{P}(C_X) $  is a product $K_j\times I$, where $K_j^{\varepsilon} =K_j\times \varepsilon, \varepsilon = 0$ or $1$, and $K_j^0\subset E_i$ and $K_j^1\subset E_{i'}$ are isomorphic  subcomplexes.  In \cite{jac-rub-0eff} we show that if $K_j^{\varepsilon}$ is not contained in a simply connected region of $E_i (E_{i'})$, then we have a properly embedded 0-weight annulus $A_j$ in $K_j\times I$. However, since $X$ is homeomorphic to $M$ and thus is an-annular, we have $i=i'$ and $A_j$ is isotopic into $E_i$.  $E_i\cup A_j$ along with  $\cup_{j\ne i}E_j$ form a barrier and we can construct a new normal surface in place of $E_i$ that is isotopic to $E_i$ but not normally isotopic.  This gives us a new collection of normal surfaces along which to consider our conditions for crushing.  Since the new surface is not normally isotopic to $E_i$, it follows from Kneser's Finiteness Theorem \cite{kne} that this can only happen a finite number of times. Hence, we eventually have a collection of normal surfaces, again called $E_1,\ldots, E_n$, where $E_i$ is parallel into $B_i$ and the combinatorial product for the component of their complement that does not meet $\cup B_i$, and again denoted $\mathbb{P}(\C_X)$, has every component where it meets  $\cup E_i$ contained in a simply connected subcomplex of some $E_i$. 

Hence,  each component $K_j\times I$ of the combinatorial product  $\bbb{P}(C_X)$ has both its end $K_j^0$ and $K_j^1$ contained in simply connected subcomplexes of $\cup E_i$, say $D_j^0$ and $D_j^1$, respectively.  While $K_j^0$ is isomorphic to $K_j^1$, it may not be the case that $D_j^0$ is isomorphic to $D_j^1$. 

Now, as in \cite{jac-rub-0eff}, we might have that $D_j^0\subset D_j^1$ (or $D_j^1\subset D_j^0$). If this is the case and we have $D_j^\varepsilon \subset E_j$, then we can construct, again using barrier surfaces, a new normal surface that is isotopic to $E_j$ but is not normally isotopic to $E_j$, arriving at a new collection of normal surfaces, still denoted $E_1,\ldots, E_n$ with $E_i$ isotopic to $B_i$. Again, by Kneser's Finiteness Theorem, this can only happen a finite number of times. 

In this way, and after a predicted number of steps, we construct a collection of normal surfaces, $E_1,\ldots, E_n$,  where we can  fill in any missing pieces in the combinatorial product $\bbb{P}(\C_X)$ to arrive at a trivial induced product region $\bbb{P}(X)$ for $X$.  For the very same reasons that $\bbb{P}(C_X) \ne X$, we have $\bbb{P}(X)\ne X$.

This takes care of product blocks in $\C_X$. We now consider truncated prisms.  If there are no cycles of truncated prisms, then we can crush the triangulation along the collection $E_1,\ldots, E_n$ constructing the desired ideal triangulation of $\open{M}$. So suppose there is a cycle of truncated prisms. Just as in \cite{jac-rub-0eff}, if the cycle is about a single edge, then there is a surgery on a member of the collection $E_1,\ldots,E_n$, giving a new collection. As before, this can only happen a finite number of times.  From the argument in \cite{jac-rub-0eff} the only possible cycle of truncated prisms about more than one edge would already be included in the induced product region. 

Hence, we can crush the triangulation $\T$ along the constructed set of normal surfaces.  The crushing gives a set of tetrahedra from the truncated tetrahedra in $\C(X)$ along with face identifications from the original face identifications, possibly translated through a chain of truncated prisms. This gives an ideal triangulation of $\open{M}$. \end{proof}

The proof of Theorem \ref{thm:construct-ideal} gives us that any time there is a closed normal surface that is isotopic into $\bdy M$, then under the hypothesis for $M$ we can crush the triangulation along a (possibly different) closed normal surface isotopic into $\bdy M$.  It seems that from this we should have a way to show that we eventually arrive at an annular efficient triangulation; however, the problem is that in crushing we arrive at an ideal triangulation of $\open{M}$ and we then need to add tetrahedra to this ideal triangulation to get back to a triangulation of the compact 3--manifold $M$.  We have not been able to show that this approach eventually terminates.  So, we switch to getting an ideal triangulation of $\open{M}$ so that the only (closed) normal surface  parallel to a vertex-linking surface is the vertex-linking surface. Then we can show any inflation of this ideal triangulation is an annular efficient triangulation of $M$. 

\begin{thm}\label{thm:ideal-annular-eff} Suppose $M$ is a compact, irreducible,
$\bdy$--irreducible, an-annular $3$--manifold. Then for any ideal triangulation $\T'$, there is a algorithm to modify the triangulation $\T'$ of $\open{M}$ to an ideal triangulation $\T^*$ of $\open{M}$ having the
property that any (closed) normal surface in $\T^*$ that is isotopic
to a vertex-linking surface is normally isotopic to that
vertex-linking surface.\end{thm}

\begin{proof}  We are given the ideal triangulation $\T' $ of $\open{M}$.  By Corollary \ref{determine-vertex-linking} we can decide if there is a closed normal surface in $\T'$ that is isotopic to a vertex-linking surface but is not itself a vertex-linking surface. If there is none, then $\T'$ satisfies the desired conclusion.  On the other hand, if there is one, then the algorithm will  construct one, say $F$, and $F$ is isotopic to a vertex-linking surface  but is not itself vertex-linking.  We wish to crush the triangulation $\T'$ along $F$.  However, to keep the situation consistent with the cell decompositions we like and our methods, if we have ideal vertex-linking surfaces  $S_{v_1^*},\ldots, S_{v_n^*}$ and notation has been chosen so that $F$ is isotopic to $S_{v_1^*}$ but is not normally isotopic to $S_{v_1^*}$, then we wish to crush the triangulation along the collection of surfaces $F,  S_{v_2^*},\ldots,  S_{v_n^*}$. 

If $X$ is the closure of the component of the complement of $F,  S_{v_2^*},\ldots,  S_{v_n^*}$ not meeting any of the ideal vertices of $\T'$, then $X$ is homeomorphic to $M$ and we can proceed in finding a collection of normal surfaces along which to crush the triangulation $\T'$, replacing the collection $E_1,\ldots, E_n$ in the proof of Theorem \ref{thm:construct-ideal}, by the collection $F,  S_{v_2^*},\ldots,  S_{v_n^*}$.  Hence, we arrive at an ideal triangulation $\T^*$ of $\open{X}$ (homeomorphic with $\open{M}$) obtained by crushing $\T'$ along a collection of normal surfaces with at least one of them not vertex-linking.  

The tetrahedra of $\T^*$ come from a subset of the tetrahedra of $\T'$ that become truncated tetrahedra in the cell decomposition $\C_X$ of $X$.  Now, since one of the normal surfaces along which we are crushing is not vertex-linking, it must contain a normal quadrilateral, and hence, at least one of the tetrahedra of $\T'$ gives a truncated prism in $\C_X$ and we have that $\abs{\T^*} < \abs{T'}$. 

It follows that the process must stop and it stops only when we have an ideal triangulation of $\open{M}$ where the only closed normal surface isotopic to a vertex-linking surface is itself a vertex-linking surface.\end{proof}

We are now ready to prove Theorem \ref{thm:annular-eff}; we  give the statement again for convenience.
	
\noindent{\bf Theorem.} \emph{Suppose $M\ne \mathbb{B}^3$ is a compact,
irreducible, $\bdy$--irreducible,  an-annular $3$--manifold with
nonempty boundary. Then there is an algorithm that will modify any triangulation of $M$  to
an annular-efficient triangulation of $M$}.

\begin{proof} Suppose $\T$ is a triangulation of $M$.  By Theorem \ref{thm:construct-ideal}, there is an algorithm to modify the triangulation $\T$  to an ideal triangulation $\T'$ of $\open{M}$.  Now by Theorem \ref{thm:ideal-annular-eff}, we can modify the ideal triangulation $\T'$ of $\open{M}$ to an ideal triangulation $\T^*$ of $\open{M}$ so that a (closed) normal surface isotopic to a vertex-linking surface is normally isotopic to that vertex-linking surface.  Construct any inflation, say $\T_{\frac{1}{2}}$,  of the ideal triangulation $\T^*$. Then by Corollary \ref{vertex-link is bdry-link}, the triangulation  $\T_{\frac{1}{2}}$ of $M$ has the property that a closed normal surface in  $\T_{\frac{1}{2}}$ that is isotopic into $\bdy M$ is a boundary-linking surface. From our observations above, the triangulation  $\T_{\frac{1}{2}}$ can not have a normal disk that is not vertex-linking ( $\T_{\frac{1}{2}}$ is 0--efficient) and can not have a normal annulus with essential boundary that is not thin edge-linking ( $\T_{\frac{1}{2}}$ is annular-efficient).\end{proof}

Notice, for 3--manifolds having connected boundary,  then for an annular-efficient triangulation, the only normal annuli are thin edge-linking.  Our original attempt was to prove for a manifold $M$ satisfying our hypothesis, then any triangulation of $M$ could be modified to one in which the only normal annuli are thin edge-linking.  What we were unable to eliminate is the possibility that the triangulation we have has a normal, \underline{compressible} annulus that is not thin edge-linking; such an annulus necessarily has boundary which is vertex-linking curves in distinct boundary components of $M$ (a ``fat annulus").  We do, however, have the following curious result.

\begin{prop} Suppose $\T$ is a 0--efficient triangulation of the compact 3--manifold $M\ne \mathbb{B}^3$.  Then for any edge $e$ of $\T$ having its vertices in distinct boundary components of $M$, a small regular neighborhood of $e$ is normally isotopic to a thin edge-linking annulus. 
\end{prop}

\begin{proof} Recall that for $M$ to have a 0--efficient triangulation, then $M$ is irreducible, $\bdy$--irreducible, all the vertices of the triangulation are in $\bdy M$, and there is precisely one-vertex  in each boundary component.  Suppose $e$ is an edge of $\T$ running between distinct boundary components of $M$.  A small regular neighborhood of $e$ is normally isotopic to a thin edge-linking annulus about $e$ if and only if there is no face of $\T$ meeting $e$ in more than one of its edges.

So, suppose there is a face $\sigma$ of $\T$ meeting $e$ in more than one of its edges.  Since $e$ runs between distinct boundary components of $M$ the only possibility is that two edges of $\sigma$ meet $e$ and $\sigma$ is a cone. Let $e'$ be the edge of $\sigma$ forming the base of the cone; i.e., $e'$ is distinct from $e$.  

The edge $e'$, which bounds a disk,  can not be in $\bdy M$ as each edge in $\bdy M$ is essential in $\bdy M$ and $M$ is $\bdy$--irreducible. If $e'$ is not in $\bdy M$,  then $e'$ bounds an embedded disk meeting $\bdy M$ at the vertices of $e$, one of which is also the vertex of $e'$. Let $v$ be the vertex of $e$ that is not a vertex of $e'$. Then using that the vertex-link of $v$ is a disk, we can truncate the cone formed by $\sigma$ and arrive at a disk $D'$ having $e'$ as its boundary and meeting $\bdy M$ only at the vertex of $e'$.  Let $N(D')$ be a small regular neighbor of $D'$. Then the frontier of $N(D')$ consists of a 2-sphere and a properly embedded disk $D$ with boundary a vertex-linking curve in $\bdy M$.  The frontier of a small regular neighborhood of $e'$ is a barrier surface and hence, $D$ can be shrunk in the complement of this small neighborhood of $e'$ to a normal disk  that is not vertex-linking ($M\ne \bbb{B}^3$). However, this contradicts that the triangulation is $0$--efficient. \end{proof}

\section{boundary slopes of surfaces}	If $S$ is a surface and $\gamma$ is a closed curve in $S$, then we call the isotopy class of $\gamma$ a \emph{slope} and refer to it as the slope of $\gamma$.  It follows from the proof of Proposition 3.2 of \cite{jac-rub-sed}  that if $M$ is a link-manifold (nonempty boundary and each boundary component is a torus) and $M$ has no essential annuli between distinct boundary components, then for any $0$--efficient triangulation $\T$ of $M$, there are only finitely many boundary slopes for normal surfaces of a bounded Euler characteristic. Hence, for such an $M$ there are only finitely many boundary slopes for incompressible and $\bdy$--incompressible surfaces of bounded Euler characteristic. We generalize this to general compact, irreducible, $\bdy$--irreducible, and an-annular 3-manifolds.

First, we have the following lemma.

\begin{lem}\label{normal-sum-edge-linking-annulus}  Suppose $M$ is a compact 3-manifold with nonempty boundary and $\T$ is a triangulation of $M$.  Furthermore, suppose $F'$ is a normal surface and $A$ is a thin edge-linking annulus about an edge in $\bdy M$. If the Haken sum $F'+A$ is defined, then $F' + A$ is either
\begin{enumerate}\item [(i)] The disjoint union $F'\cup A$, \item[(ii)] A normal surface $F$ and a vertex-linking surface, or
\item[(iii)] A normal surface $F$ isotopic to $F'$.\end{enumerate}\end{lem}
\begin{proof} Following a small isotopy of $A$, we have that $F'\cap A$ is at most a finite number of normal spanning arcs running through the normal quads of $A$ (hence, only meeting normal triangles of $F'$).  
If $F' \cap A=\emptyset$, then we have conclusion $(i)$. So, assume $F'\cap A\ne \emptyset$.

\begin{figure}[htbp] \psfrag{A}{\scriptsize
{$A$}}
 \psfrag{F}{\scriptsize{$F'$}}\psfrag{e}{\tiny{$e\in \bdy M$}}
 \psfrag{V}{\scriptsize{$v$}}\psfrag{D}{\scriptsize{$D$}}\psfrag{G}{\scriptsize{$F$}}
 \psfrag{1}{\scriptsize{$F'+A = D+F$}}\psfrag{2}{\scriptsize{$F'+A = F \sim F'$}}\psfrag{E}{\scriptsize{$F'+A=F$}}

        \begin{center}
\epsfxsize =3.5 in \epsfbox{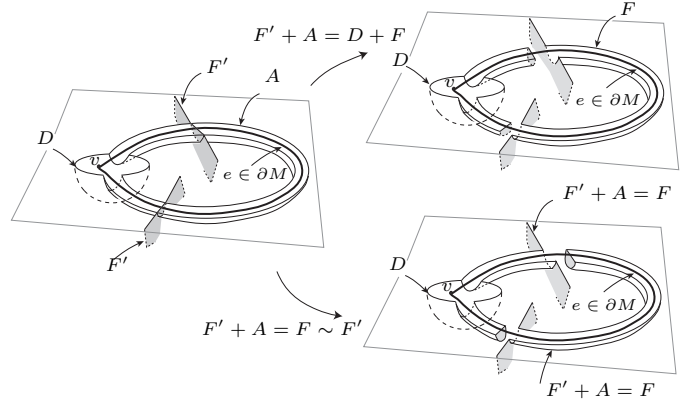} \caption{Haken sum of normal surface with thin edge-linking annulus. } \label{fr-thin-annulus-Haken-sum.eps}
\end{center}
\end{figure}
 
 An arc $\alpha\subset F'\cap A$ cuts off a small disk $d_{\alpha}$ in $F$ where $\bdy d = \alpha\cup\beta$ with $\beta\subset\bdy M$.  Following a regular exchange along $\alpha$, a copy of $d_\alpha$ is joined with $A$ forming a $\bdy$-compression of $A$. See Figure \ref{fr-thin-annulus-Haken-sum.eps}. Since $F'+A$ is a normal surface, it is only possible that two such $\bdy$--compressions occur adjacent on $A$ when together they cut off the vertex-linking disk. See the top right-hand drawing in Figure \ref{fr-thin-annulus-Haken-sum.eps}. This gives possibility (ii) of our conclusion.  Otherwise, we get possibility (iii). See bottom right-hand drawing in Figure \ref{fr-thin-annulus-Haken-sum.eps}. \end{proof}

\begin{thm} Suppose $\T$ is an annular efficient triangulation of the compact 3--manifold $M$.  Then there are only finitely many boundary slopes for connected normal surfaces in $\T$ of bounded Euler characteristic.\end{thm} 

\begin{proof}  A triangulation $\T$ determines a collection of normal surfaces. Among these is a unique collection of normal surfaces (the fundamental normal surfaces) $F_1,\ldots, F_K,$ $T_1,\ldots,T_M,$ $ A_1,\ldots, A_N$ such that any normal surface $F$ can be written as a Haken sum $F = \sum_1^K p_k F_k +\sum_1^M q_m T_m +\sum_1^N r_n A_n$, where $p_k, 1\le k\le K; q_m, 1\le m\le M;$ and $r_n, 1\le n\le N$ are nonnegative integers, $\chi(F_k)<0$, $T_m$ a torus or Klein bottle, and $A_n$ an annulus. Since $\T$ is annular-efficient, $M$ is an-annular and, hence, no $A_m$ can be a M\"obius band.  If we set $F'' = \sum_1^K p_k F_k$, then we have $\chi(F) = \chi(F'')$ and observe for surfaces $F$ with bounded Euler characteristic, there can be only finitely many sums $F'' = \sum_1^K p_k F_k$ of bounded Euler characteristic. 

Hence, for bounded Euler characteristic, there are at most a finite number of boundary slopes for surfaces of the form   $F' = \sum_1^K p_k F_k +\sum_1^M q_m T_m$. However, for any connected surface $F = F'+A_n$, only one of the possibilities in Lemma \ref{normal-sum-edge-linking-annulus} can hold and that is (iii). It follows that for $F$ connected and $F = F' + \sum_1^N r_n A_n$, we $F\sim F'$ and so $F$ and $F'$ have the same boundary slopes. This proves our theorem.\end{proof}

The following corollary is immediate as an incompressible and $\bdy$--incompressible surface in an irreducible and $\bdy$--irreducible 3--manifold is isotopic to a normal surface in any triangulation.

\begin{cor} Suppose $M\ne \mathbb{B}^3$ is a compact,
irreducible, $\bdy$--irreducible,  an-annular $3$--manifold. Then there are only finitely many boundary slopes for connected, incompressible, and $\bdy$--incompressible  surfaces in $M$ of bounded Euler characteristic
\end{cor}

\end{document}